\documentclass[11pt]{amsart}
\usepackage{amsmath}
\usepackage[letterpaper,margin=1in]{geometry}
\usepackage{amsbsy,amsfonts,amsmath,amssymb,amscd,amsthm,mathrsfs}
\usepackage{enumitem,graphics,epsfig,color,ifpdf}
\usepackage{subfig}
\usepackage{oldgerm}
\usepackage{mathabx}
\usepackage{mleftright}
\usepackage{MnSymbol}
\usepackage{tikz-cd}
\usepackage{color}
\normalsize
\makeatletter
\newcommand{\mathleft}{\@fleqntrue\@mathmargin0pt}
\newcommand{\mathcenter}{\@fleqnfalse}
\makeatother
\newtheorem{theorem}{Theorem}[section]
\newtheorem{lemma}[theorem]{Lemma}
\newtheorem{definition}[theorem]{Definition}
\newtheorem{corollary}[theorem]{Corollary}
\newtheorem{remark}[theorem]{Remark}

\newtheorem{proposition}[theorem]{Proposition}

\newtheorem{example}[theorem]{Example}
\newtheorem{thmx}{Theorem}

\newtheorem*{corollary C}{Corollary C}
\newtheorem*{theorem B}{Theorem B}

\DeclareMathOperator{\interior}{\text{int}}
\DeclareMathOperator{\A}{\textbf{\textit{a}}}
\DeclareMathOperator{\B}{\textbf{\textit{b}}}
\DeclareMathOperator{\C}{\textbf{\textit{c}}}

\DeclareMathOperator{\U}{\textbf{\textit{u}}}
\DeclareMathOperator{\V}{\textbf{\textit{v}}}

\DeclareMathOperator{\const}{\text{const.}}

\begin{document}
\title{Absolutely Continuous Invariant Measure for Generalized Horseshoe Maps}
\author{Abbas Fakhari and Maryam Khalaj}
\address{Department of Mathematics, Shahid Beheshti university, Tehran 19839, Iran}
\email{a\_fakhari@sbu.ac.ir}
\address{Department of Mathematics, Shahid Beheshti university, Tehran 19839, Iran}
\email{m\_khalaj@sbu.ac.ir}
\begin{abstract}
In this paper, we study the SRB measures of generalized horseshoe map. We prove that under the conditions of transversality and fatness, the SRB measure is actually absolutely continuous with respect to the Lebesgue measure.
\end{abstract}

%\keywords{generalized horseshoe map, fatness, transversality, absolutely continuous invariant measure}

%\keywords[2020 Mathematics Subject Classification]{\codes[Primary]{37CX}\codes[Secondary]{37DX}}
\maketitle
\section{Introduction}
The transversality condition was introduced in the 90s to calculate the Hausdorff dimension of the self-similar sets and to find absolutely continuous invariant measures.
Roughly speaking, two families $\mathcal{F}$ and $\mathcal{G}$ of curves in $\mathbb{R}^2$ are transverse if {\em almost all} pair $(f,g)$, with $f\in\mathcal{F}$ and $g\in\mathcal{G}$ are transversal with uniform slope. Under the transversality condition, Pollicott and Simon \cite{PS} determined the Hausdorff dimension of the missing digit sets $$\Lambda(\lambda)=\{\sum_{k=1}^\infty i_k\lambda^k:i_k=0,1,3\}$$ for almost every $\lambda\in(\frac{1}{4},\frac{1}{3})$.
Solomyak \cite{S2} proved the Lebesgue measure of $\Lambda(\lambda)$, for almost every $\lambda>1/3$, is positive provided the transversality condition holds. Motivated by this scheme, it is shown that if the transversality condition holds for a certain contracting affine iterated function system in $\mathbb{R}^d$, $d\geq 2$, then the Hausdorff dimension of the attractor is the minimum of $d$ and the singularity dimension \cite{JPS}.

The first motivating classical example in higher dimension is {\em generalized baker map} $B_\lambda:[-1,1]^2\to [-1,1]^2$ defined by
$$B_\lambda(x,y)=\begin{cases} (2x-1,\lambda y+1-\lambda) \quad & x\geq 0\\
(2x+1,\lambda y-1+\lambda) \quad & x< 0.\end{cases}$$
The transversality condition implies that for almost every $\lambda\in (\frac{1}{2},1]$, map $B_\lambda$ admits an absolutely continuous ergodic measure \cite{S2}. Generalizing Baker maps, Tsujii \cite{T} showed that the SRB measure of any transversal solenoidal attractor $T:S^1\times \mathbb{R}\to S^1\times \mathbb{R}$ defined by
$$T(x,y)=(\ell x,\lambda y+f(x)),$$
is absolutley continuous, where $f$ is a $C^2$ function, $\ell\geq2$ a natural number, $0<\lambda<1$ and $\ell\lambda>1$. After that, inspiring Tsujii's strategy, Rams \cite{R} provided a geometric approach to prove the absolute continuity of the SRB measure for a generalized map $T:S^1\times \mathbb{R}^d\to S^1\times \mathbb{R}^d$ defined by $$T(x,y)=(f(x),g(x,y)),$$ where $f$ is a $k$ to $1$ expanding map and $g$ is a contraction.

In this paper, we study a class of dynamical systems having the most extended structure called {\em generalized horseshoe maps}. The generalized horseshoe map initially defined by Jakobson and Newhouse in \cite{JN1} to detect the SRB measure in the most general case.
The generalized horseshoe map is a piecewise hyperbolic map defined on a countable family of vertical strips.
We show that the two assumptions of area-expanding and the transversality of the unstable manifolds lead to absolute continuity of the SRB measure.
\begin{thmx}\label{T1}
For any transversal fat generalized horseshoe map, the SRB measure is absolutlely continuous with respect to the Lebesgue measure.
\end{thmx}
Theorem \ref{T1} has novelties in some ways. The continuity of the map is removed and, unlike the classical case of skew-product, the boundary of strips have non-zero curvature. These conditions accompanied by the infinity of the strips force further calculation for adaptation.

The generalized horseshoe map is defined in Section 2. Section 3 is devoted to the proof of the existence of an SRB measure. The precise definition of the fatness and transversality conditions are presented in Subsection 4.1.  Subsection 4.2 consists of the essential lemmas to control the distortions. Finally, the absolute continuity of the SRB measure is shown in Subsection 4.3.
\section{Generalized Horseshoe Map (GHM)}\label{s2}
In this section, we introduce the model we are dealing with in this paper. Suppose that $\{S_1,S_2,\cdots\}$ is a countable collection of closed curvilinear rectangles in $S=[0,1]^2$ whose interiors are non-overlapping and covering $S$ up to a subset of zero Lebesgue measure. Each $S_i$ is full height whose left and right boundaries are graphs of smooth functions. Let $F_i=(F_{i1},F_{i2})$ be a $C^2$ diffeomorphism on $S_i$ and let $U_i$ be the image of $S_i$ under $F_i$ which is full width and bounded by graphs of smooth functions from top and bottom.
Suppose that for constants $0<\alpha<1$ and $K_0>1$, the map $F$ on $S$ given by $F|_{\interior  S_i}=F_i$ satisfies the \textit{hyperbolicity conditions} described in \cite{JN}, so the following cone conditions hold:
\begin{enumerate}[label=\textbf{H\arabic*}]
\item $DF(\mathcal{C}_\alpha^u)\subseteq \mathcal{C}_{\alpha}^u$ and $DF^{-1}(\mathcal{C}_\alpha^s)\subseteq \mathcal{C}_{\alpha}^s$,
\item \label{H2} $|DF(v)|\geq K_0|v|,\,\text{for}\,v\in \mathcal{C}^u_\alpha$ and $|DF^{-1}(v)|\geq K_0|v|,\,\text{for}\,v\in \mathcal{C}^s_\alpha$.
\end{enumerate}
%\begin{equation}
%DF(\mathcal{C}_\alpha^u)\subseteq \mathcal{C}_{\alpha}^u
%\end{equation}
%\begin{equation}\label{2}
%|DF(v)|\geq K_0|v| \quad \text{for} \quad v\in \mathcal{C}^u_\alpha
%\end{equation}
%\begin{equation}\label{3}
%|DF^{-1}(v)|\geq K_0|v| \quad \text{for} \quad v\in \mathcal{C}^s_\alpha
%\end{equation}
%\begin{equation}
%DF^{-1}(\mathcal{C}_\alpha^s)\subseteq \mathcal{C}_{\alpha}^s
%\end{equation}
where $\mathcal{C}_\alpha^s=\{(v_1,v_2):|v_1|\leq \alpha|v_2|\}$ and $\mathcal{C}_\alpha^u=\{(v_1,v_2):|v_2|\leq \alpha|v_1|\}$, and we use the max norm i.e. $|v|=|(v_1,v_2)|=\max\{|v_1|,|v_2|\}$. The map $F$ with the above conditions is called the \textit{generalized horseshoe map}. Also, for each $i$ and $z\in S_i$, the hyperbolic conditions yield \cite[Lemma 4.1]{JN}

\begin{equation}\label{5}
\frac{|F_{i1y}(z)|}{|F_{i1x}(z)|}\leq \alpha,
\end{equation}
\begin{equation}\label{6}
\frac{|F_{i2x}(z)|}{|F_{i1x}(z)|}\leq \alpha,
\end{equation}
\begin{equation}\label{7}
\frac{|F_{i2y}(z)|}{|F_{i1x}(z)|}\leq \frac{1}{K_0^2}+\alpha^2.
\end{equation}
Let $\Sigma_\mathbb{N}^\infty:=\{(a_i)_{i=1}^\infty\mid a_i\in\mathbb{N}\}$, $[\A]_n:=(a_i)_{i=1}^n$ and $F_{[\A]_n}:=F_{a_n}\circ\cdots\circ F_{a_1}$.

The stable and unstable manifolds can be described in two following approaches.

\begin{itemize}
\item {\it Analytical Definition.}~\\
For any $X=(x_n)_{n\in\mathbb{Z}}$ in the inverse limit space $\overleftarrow{M}=\{(x_n)_{n\in\mathbb{Z}},\,\, F(x_n)=x_{n+1}\}$, put
$$E^u(X)=\bigcap_{n\geq 0} DF^{n}(x_{-n})(\mathcal{C}^u_\alpha(x_{-n})),$$
$$E^s(X)=\bigcap _{n\geq 0} DF^{-n}(x_n)(\mathcal{C}^s_\alpha(x_n)).$$
$E^s(X)$ and $E^u(X)$ are stable and unstable directions at $X$. By the definition, $E^s(X)$ only depends on the 0th position of $X$. By Hadamard-Perron Theorem, directions $E^s$ and $E^u$ are integrable (see \cite{P} for a complete discussion).
\item {\it Geometrical Definition.}~\\
For any finite word $[\A]_n$, put
$$S_{[\A]_n}:=S_{a_1}\cap F_{a_1}^{-1}(S_{a_2\ldots a_n})$$
and $U_{[\A]_n}:=F_{[\A]_n}(S_{[\A]_n})$ (see Figure \ref{GH1}). For any infinite word $\A\in \Sigma_\mathbb{N}^\infty$, put
$$W^s_{\A}:=\bigcap_{n\geq 1} S_{[\A]_n}, \quad W^u_{\A}:=\bigcap_{n\geq 1} U_{[\A]_n}.$$
Actually, the set $\Lambda:=\bigcup_{\A}\bigcap_{n\geq1}U_{[\A]_n}$ defines a topological attractor for $F$. The stable and unstable manifolds $W_{\A}^s$ and $W_{\A}^u$ are graphs of $C^1$ functions defined in any point of $\Lambda$. Note that any point of $\Lambda$ has a unique stable manifold and, probably, non-unique unstable manifold.
\end{itemize}
\begin{figure}
\def\svgwidth{7cm}
\def\svgwidth{7cm}
\includegraphics{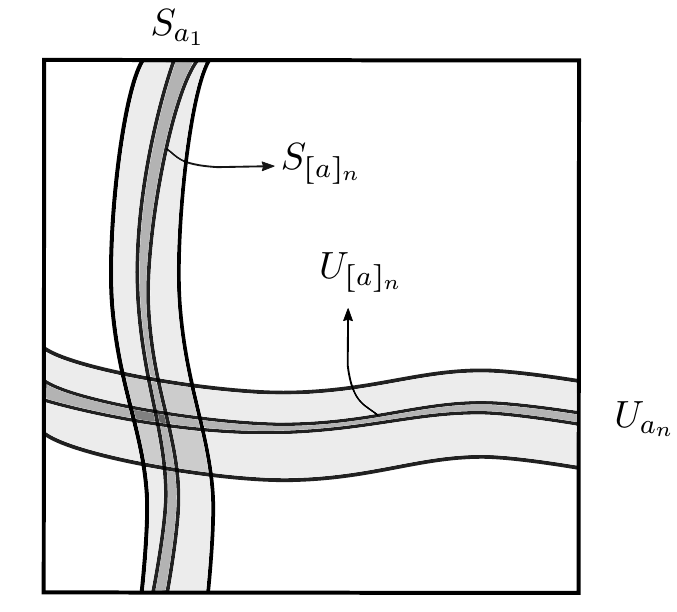}
\centering
\caption{Stable and Unstable Strips in a Generalized Horseshoe Map}
\label{GH1}
\end{figure}
\section{SRB measure for GHM}
There are classical approaches for finding SRB measures for Anosov endomorphisms which are not applicable in our case because of the existence of discontinuities in the GHM. Jakobson and Newhouse have already proved that the GHM has an SRB measure \cite{JN}. Salas has tried in \cite{S} to present a shorter proof, however, the proof seems not to be complete.

In this section, we provide a short proof for the existence of an SRB measure essentially based on the absolute continuity of the stable manifolds. Roughly speaking, our method is to find an invariant measure whose disintegration along unstable manifolds is equivalent to the Lebesgue measure. First, we study the special case of skew-products inspired by Tsujii in \cite{T} and then, we deal with the general case in which the stable manifolds may have non-zero curvature.

\subsubsection{Absolute Continuity of Stable Manifolds}
The main tool to find the SRB measure is the absolute continuity of the stable manifolds which is proved in \cite{JN}.

Let $\tilde{S}_0=\bigcup_i\interior(S_i)$ and define inductively $\tilde{S}_n=\tilde{S}_0\cap F^{-1} (\tilde{S}_{n-1})$, for $n\geq1$. Let $\tilde{S}=\bigcap_{n\geq0}\tilde{S}_n$. Then $\tilde{S}$ is a $F$-invariant full Lebesgue measure set. Clearly $\tilde{S}$ is $W^s$-saturated, where $W^s$ is the stable lamination. Let $D_1$ and $D_2$ be two disks in $S$ transverse to the stable lamination and $H_{D_1,D_2}$ be the holonomy map along the stable lamination, that is $H_{D_1,D_2}:D_1\cap \tilde{S}\to D_2\cap \tilde{S}$
$$H_{D_1,D_2}(z_1):=W^s(z_1)\cap D_2,$$
where $W^s(z_1)$ is the unique stable leaf of $W^s$ through $z_1$. Suppose that $m_{D}$ is the Lebesgue measure induced by the Riemannian metric on the disk $D$. Jakobson and Newhouse used some geometric and distortion conditions to prove the absolute continuity of stable lamination, meaning that the measure $\tilde{m}_{D_2}=(H_{D_1,D_2})_*m_{D_1}$ is equivalent to $m_{D_2}$. Hence, there is a measurable map $J:D_2\to[0,+\infty)$ which is integrable with respect to $\tilde{m}_{D_2}$ such that for any Borel set $A\subseteq D_2$, we have
$$m_{D_2}(A)=\int_AJd\tilde{m}_{D_2}.$$

\subsubsection{Lifting to an SRB measure}
Here, we use the classical lifting procedure and absolute continuity of stable lamination to generate an SRB measure for $F$ (see also \cite{V}).

Suppose that $p^s:\tilde{S}\to[0,1]$ is the projection along the stable leaves and let $I_i=p^s(S_i)$. Define $g:\bigcup_iI_i\to[0,1]$ by
$$g(x):=p^s\circ F(x,0).$$
In this case, $g(I_i)=[0,1]$, for each $i$, and also $p^s\circ F=g\circ p^s$ on $\tilde{S}$. Since the stable lamination is absolutely continuous and the stable leaves are of codimension one, Jacobian $J$ of the holonomy map is essentially smooth (see \cite{Ro}), this leads $g$ to be piecewise expanding. According to Folklore theorem \cite{A,W}, $g$ has an absolutely continuous invariant probability measure, ACIP, say $\mu_g$.
\begin{proposition}
The map $F$ has an SRB measure.
\end{proposition}
\begin{proof}
For a given continuous function $\psi:S\to \mathbb{R}$, let $\overline{\psi}:[0,1]\to\mathbb{R}$ defined by $\overline{\psi}(x)=\psi(x,0)$. Then we claim that $$\lim_{n\rightarrow\infty}\int(\overline{\psi\circ F^n})d\mu_g$$ exists. For given $\epsilon>0$, there exists $\delta>0$ such that for any $z_1$ and $z_2$ in $S$ satisfying $|z_1-z_2|<\delta$, we have $|\psi(z_1)-\psi(z_2)|<\epsilon$. Since $F$ is contracting along the stable manifolds, there exists $n_0\geq0$ such that for any $n+k\geq n\geq n_0$ we have $|F^{n+k}(x,0)-F^n\circ p^s\circ F^k(x,0)|<\delta$. Therefore
\begin{align*}
\left|\int(\overline{\psi\circ F^{n+k}})d\mu_g \hspace{-.1cm}-\hspace{-.1cm}\int(\overline{\psi\circ F^n})d\mu_g \right| &=\left|\int(\psi\circ F^{n+k})(x,0)d\mu_g \hspace{-.1cm}-\hspace{-.1cm}\int(\overline{\psi\circ F^n})\circ g^k d\mu_g \right|\\
&=\left|\int(\psi\circ F^{n+k})(x,0)d\mu_g\hspace{-.1cm}-\hspace{-.1cm}\hspace{-.1cm}\int(\psi\circ F^n\circ p^s\circ F^k)(x,0) d\mu_g \right|\\
&\leq\epsilon.
\end{align*}
The first equality holds by the $g$-invariance of $\mu_g$. So we have shown that $\{\int(\overline{\psi\circ F^n})d\mu_g\}$ is a Cauchy sequence in $\mathbb{R}$ and it converges. Define
$$\hat{\mu}(\psi)=\lim_{n\rightarrow\infty}\int(\overline{\psi\circ F^n})d\mu_g.$$
Obviously $\hat{\mu}$ is a linear operator on the space of continuous functions $\psi:S\to \mathbb{R}$. Also $\hat{\mu}(1)=1$ and $\hat{\mu}$ is non-negative that is $\hat{\mu}(\psi)\geq0$ for $\psi\geq0$. So, by Riesz representation theorem, there exists a unique measure called $\mu_F$ such that for any continuous map $\psi$
$$\hat{\mu}(\psi)=\int \psi d\mu_F.$$
By the definition of $\hat{\mu}$, $\mu_F$ is $F$-invariant and the disintegration of $\mu_F$ along any unstable manifold $W^u$ is the pullback of $\mu_g$ by the holonomy map $H_{W^u}:W^u\to [0,1]$.
\end{proof}
\section{Absolute Continuity of the SRB Measure}
Our approach for proving the absolute continuity of $\mu_F$ is based on two \textit{general assumptions} of transversality and fatness which are appeared in a more specific way in \cite{R, T}. In this context, some more \textit{special assumptions} are needed for the simplicity of calculations. Recall that $F_i(x,y)=(F_{i1}(x,y),F_{i2}(x,y))$. Put
$$|D^2F_i(z)|:=\max_{j=1,2,(k,l)=(x,x),(x,y),(y,y)}\{|F_{ijkl}(z)|\}.$$
We assume that
\begin{enumerate}[label=\textbf{A\arabic*}]
\item \label{A1}there exists a constant $C_0>0$ such that $\sup_{i\geq1}\sup_{z\in S_i} |D^2F_i(z)|<C_0$,
\item \label{A2}$J_{F_i}:=F_{i1x}F_{i2y}-F_{i1y}F_{i2x}$ and $\sup_{i\geq1}J_{F_i}<\infty$,
\item \label{A3}$\sup_{i\geq1}\sup_{z\in S_i, F_i(w)=z} (F_{i1y}(z)F_{i2x}(w))/F_{i2y}(z)<\infty$,
\item \label{A4}$|F_{i1y}(z)|,|F_{i2x}(z)|<1/8$ for any $z\in\tilde{S}$ and $i\geq1$.
\end{enumerate}
According to \cite[Lemma 4.2]{JN} and using \ref{A1}, there exists a positive constant $C_1$ such that for any two close points $z$ and $w$ lying on an unstable piece in $S_i$,
\begin{equation}\label{8}
\frac{|F_{i1x}(z)|}{|F_{i1x}(w)|}\leq \exp(C_1),
\end{equation}
where $C_1=\sqrt{2}(1+\alpha) C_0$.
Theorem \ref{T1} will be proven by using these general and special assumptions.
\subsection{Fatness and Transversality Conditions}
Let $J$ be an interval strictly containing $I=[0,1]$ and $\hat{S}=[0,1]\times J$. Suppose that $\hat{S}_i$ and $\hat{U}_i$ are the neighborhoods of $S_i$ and $U_i$ in $\hat{S}$ respectively such that each $F_i$ can be extended to a $C^2$ hyperbolic diffeomorphism $\hat{F}_i:\hat{S}_i\to\hat{U}_i$ which has the same properties as $F_i$ and $p^s(\hat{S}_i)=I_i$. For any word $\A\in \Sigma_\mathbb{N}^\infty$, define $\hat{U}_{[\A]_n}$ in a similar way as $U_{[\A]_n}$. Denote the intersection of $U_{[\A]_n}$ with the stable manifold $W^s(x)$ through the point $(x,0)$ by $U_{[\A]_n}(x)$. Define the notation $\hat{U}_{[\A]_n}(x)$ in a similar way. Let $d([\A]_n)=\max_{x}|\hat{U}_{[\A]_n}(x)|$.
\begin{definition}
The generalized horseshoe $F$ is called fat if there exist $K_1,\epsilon>0$ such that
$$ |I_{[\A]_n}|\leq K_1(d([\A]_n))^{1+\epsilon},$$
holds for all $[\A]_n\in\Sigma_\mathbb{N}^n$, where $I_{[\A]_n}=p^s(S_{[\A]_n})$.
\end{definition}

To define the transversality condition, we need a bit of notation. For any $x\in\bigcup I_i$ and $\A\in\Sigma_\mathbb{N}^\infty$, put $W^u_{\A}(x):=W^u_{\A}\cap W^s(x)$. For $\delta>0$, two words $\A,\B\in\Sigma_\mathbb{N}^\infty$ are $\delta$-\textit{transversal} if
$$d_s(W^u_{\A}(x),W^u_{\B}(x))>\delta\quad \text{or}\quad \left|\frac{d}{dx}W^u_{\A}(x)-\frac{d}{dx}W^u_{\B}(x)\right|>\delta$$
holds for almost all $x$, where $d_s$ is the metric induced by the Riemannian metric on the stable manifold. Two finite words $[\A]_n$ and $[\B]_m$ are $\delta$-transversal if for any  $\U,\V\in\Sigma_\mathbb{N}^\infty$, the two infinite words $[\A]_n\U$ and $[\B]_m\V$ are $\delta$-transversal. Note that for any two $\delta$-transversal finite words $[\A]_n$ and $[\B]_m$, if $a_1\neq b_1$ then
\begin{equation}\label{vol}
\mathrm{vol}(\hat{U}_{[\A]_n}\cap \hat{U}_{[\B]_m})\leq \delta^{-1}d([\A]_n)d([\B]_m).
\end{equation}

For any $r<|J|$ and $\A\in \Sigma_\mathbb{N}^\infty$, let $n$ be the biggest number satisfying $d([\A]_n)\geq r$. Let $M(r)$ be the set of all such finite words $[\A]_n$. Put $${\bf M}_\delta^{Tr}(r):=\{([\A]_n,[\B]_m)\in M(r)^2; \; [\A]_n \; \text{and} \; [\B]_m \; \text{are}\; \delta \; \text{-transversal}\}$$
and ${\bf M}_\delta^{NTr}(r)$ be the complement of ${\bf M}_\delta^{Tr}(r)$.
\begin{definition}\label{4.2}
The map $F$ satisfies the transversality condition if for some $\delta>0$
$$\limsup_{r\to0}r^{-2}\sum_{([\A]_n,[\B]_m)\in {\bf M}_\delta^{NTr}(r)}\mathrm{vol}\;(\hat{U}_{[\A]_n}\cap\hat{U}_{[\B]_m}) |I_{[\A]_n}| |I_{[\B]_m}|<\infty.$$
\end{definition}
%\begin{remark}
%In the case of finite strips, since the contraction and expansion along the stable and unstable manifolds are uniformly bounded from above and below, the transversality condition can be presented as follows
%$$\lim_{\delta\to0}\limsup_{r\to0}r \#{\bf M}_r^{NTr}(\delta)<\infty$$
%\end{remark}
%\begin{figure}
%\def\svgwidth{15cm}
%\def\svgwidth{10cm}
%\includegraphics{Transverse.pdf}
%\centering
%\caption{Transversality of $\hat{U}_{[a]_n}$ and $\hat{U}_{[b]_m}$}
%\label{GH3}
%\end{figure}
\begin{example}
\normalfont
Let $I_1=[0,1/2]$ and $I_2=[1/2,1]$ and for $1/2<b<a<1$, consider piecewise affine map $F=(f_1,f_2)$ on $[0,1]^2$ with
\begin{align*}
f_1(x,y)&=(2x,(a+2x(b-a)y+(1-a)2x(a-b)),\,\,\text{for} (x,y)\in I_1\times[0,1],\\
f_2(x,y)&=(2x-1,(a+(2x-1)(b-a))y),\hspace{1.5cm}\text{for}(x,y)\in I_2\times[0,1],
\end{align*}
see Figure  \ref{affine} below.
\begin{figure}[ht]
\centering
%\subfloat[n=1]
{\includegraphics[width=5cm]{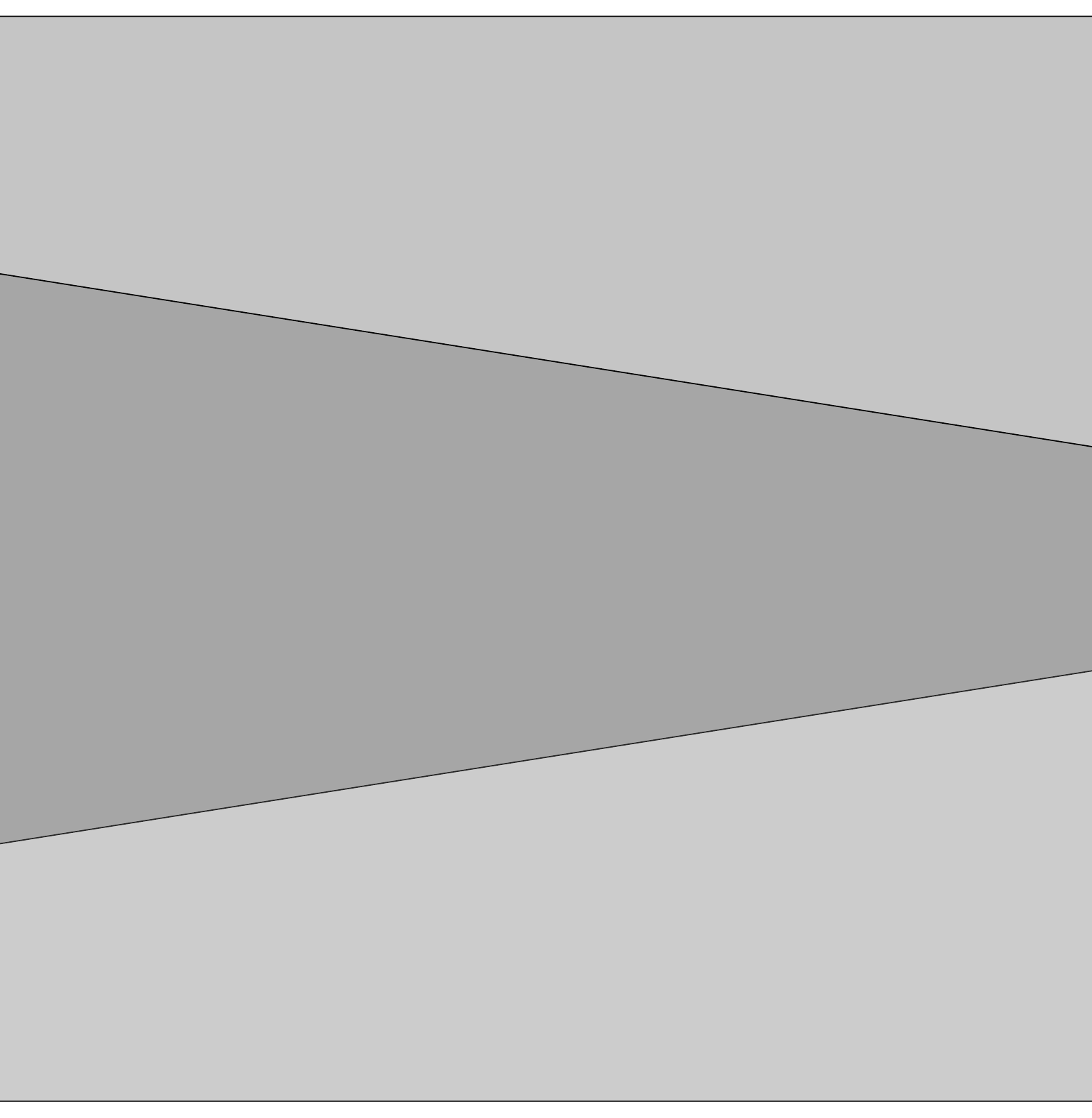}}
\qquad
%\subfloat[n=6]
{\includegraphics[width=5cm]{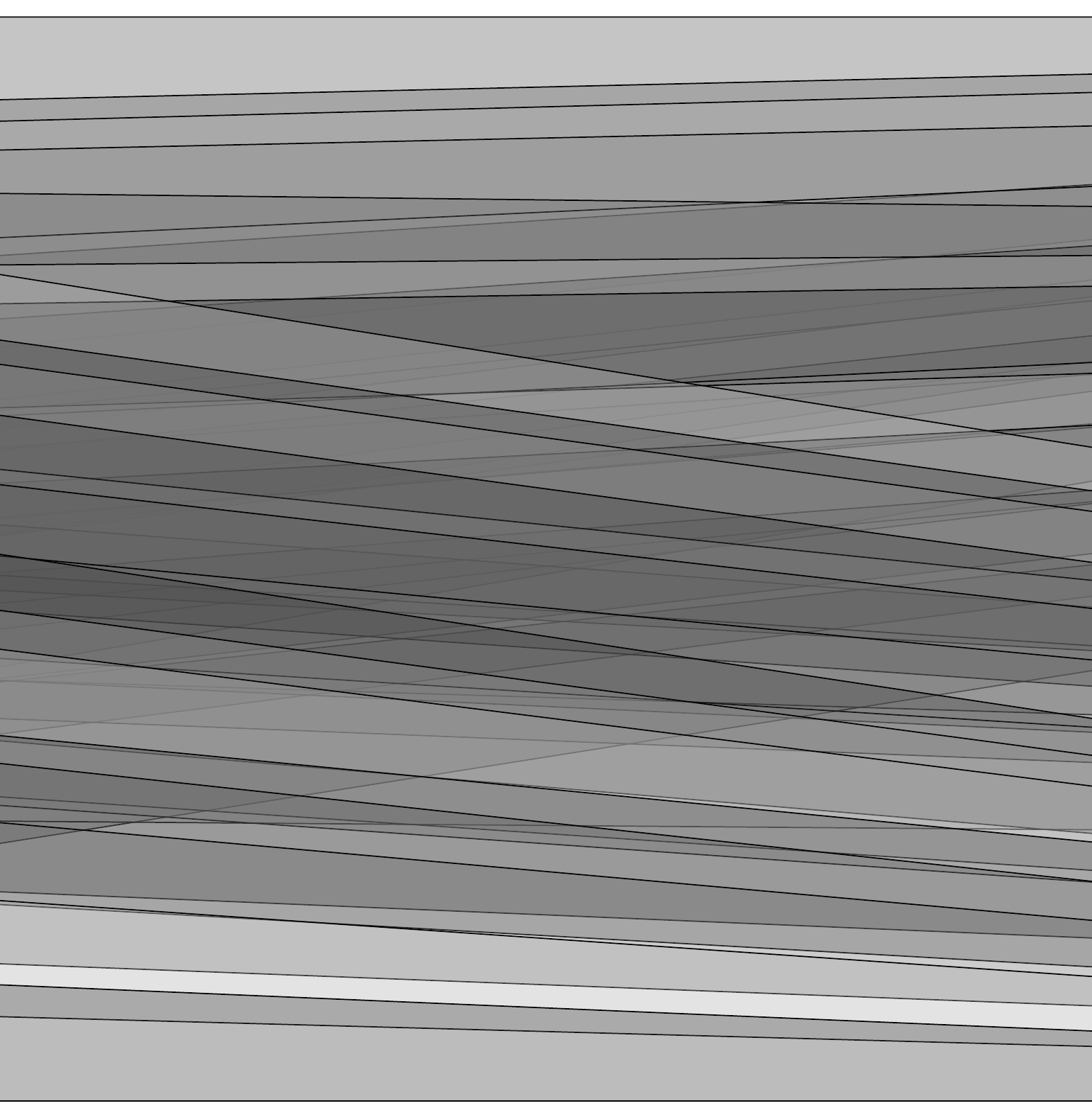}}
\caption{Image of $F^n$ for $a=0.8$ and $b=0.55$. Left, for $n=1$ and right for $n=5$.}\label{affine}
\end{figure}\\
Put 
$$\gamma:=\inf_{f_1(z)=f_2(w)}\angle(D_xf_1(\mathcal{C}^u_\alpha(z)),D_yf_2(\mathcal{C}^u_\alpha(w)))>\frac{a-b}{2}.$$
Choose $m_0\in\mathbb{N}$ such that $(a/2)^{m_0}(\gamma)\sim\delta$. 
By the construction, for $[\A]_n\in M(r)$, $n\sim \mathcal{O}(\log r/\log a)$ and 
$d([\A]_m)\leq \dfrac{r}{b^{n-m}}$. By the construction, if $([\A]_n,[\B]_n)\in {\bf M}^{NTr}_\delta(r)$ then $a_j=b_j$, for $j=m_0,\ldots,n$. For $m\leq m_0$, put
$${\bf M}^{NTr(m)}_\delta(r)=\{([\A]_n,[\B]_n)\in {\bf M}^{NTr}_\delta(r):a_j=b_j,\; j=m,\ldots,n\}.$$
 By \eqref{vol}, for $([\A]_n,[\B]_n)\in {\bf M}^{NTr(m)}_\delta(r)$ we have $\mathrm{vol}(U_{[\A]_m}\cap U_{[\B]_m})\leq \delta^{-1} (\dfrac{r}{b^{n-m}})^2$ and hence
\begin{align*}
\mathrm{vol}(U_{[\A]_n} \cap U_{[\B]_n})&\leq \delta^{-1} (\dfrac{r}{b^{n-m}})^22^{-(n-m)}\max (J_{f_1^{n-m}},J_{f_2^{n-m}})\\
&\leq \delta^{-1} (\dfrac{r}{b^{n-m}})^22^{-(n-m)}(2a)^{n-m} =\delta^{-1} r^2(\dfrac{a}{b^2})^{n-m}.
\end{align*}
Then,
\begin{align*}
r^{-2}&\sum_{([\A]_n,[\B]_n)\in {\bf M}_\delta^{NTr}(r)}\mathrm{vol}\;({U}_{[\A]_n}\cap{U}_{[\B]_n}) |I_{[\A]_n}| |I_{[\B]_n}|\\
&=r^{-2}\sum_{m=1}^{m_0}\sum_{([\A]_n,[\B]_n)\in {\bf M}_\delta^{NTr(m)}(r)}\mathrm{vol}\;({U}_{[\A]_n}\cap{U}_{[\B]_n}) |I_{[\A]_n}| |I_{[\B]_n}|\\
&\leq \delta^{-1} 4^{-n}\sum_{m=1}^{m_0}(\frac{a}{b^2})^{n-m}\\
&\leq \delta^{-1} r^{(\log \dfrac{a}{4b^2})/(\log a)},
\end{align*}
which converges.
\end{example}
\subsection{Control of Distortions}
In the following, we describe the distortion properties of the GHM through some propositions and lemmas. The easy case is the control of distortion along the stable manifolds comes from classical distortion control.
\begin{proposition}(Bounded Distortion Property Along Stable Manifolds)\label{R1}
There exists $K>0$ such that for any $[\A]_n$ and any $z,w\in S$ belonging to a small stable piece in $U_{[\A]_n}$,
$$K^{-1}<\frac{|D_z^sF_{[\A]_i}^{-1}|}{|D_w^sF_{[\A]_i}^{-1}|}<K \quad  \text{for any}\; i=1,\ldots,n,$$
where $D^s$ is the derivative along the stable manifolds.
\end{proposition}
The next proposition is actually Lemma 2.2 in \cite{R}. However, the two generalizations on the model, the infinity of the strips and the non-zero curvature of stable manifolds force more delicate details to prove.
\begin{proposition}\label{pro4}
Suppose that $[\A]_n$ is an arbitrary finite word. Then there exists a constant $K_2$ such that for almost all $x_1$ and $x_2$ in $[0,1]$ we have
$$K_2^{-1}<\frac{|U_{[\A]_n}(x_1)|}{|U_{[\A]_n}(x_2)|}<K_2$$
and similarly for $\hat{U}_{[\A]_n}$. This constant is independent of the choice of $[\A]_n$.
\end{proposition}
Proposition \ref{pro4} and the following discussion are the key ingredients needed to prove our main goal. For the proof of Proposition \ref{pro4}, we need more subtle details. First, we change the coordinate from the standard one to the coordinates induced by splitting $E^u_z\oplus E^s_z$.

To simplify, abuse of notation, denote diffeomorphism $F_{a_i}:S_{a_i}\to U_{a_i}$, for $i=1,\ldots,n$ by $F$. For $z\in \tilde{S}$, and the unstable curve $\gamma$ through $z$, consider splitting $\mathbb{R}^2=E^u_z\oplus E^s_z$, where $E^u_z$ contains the unit vector tangent to $\gamma$ at $z$ and $E^s_z$ is tangent to the local stable manifold through $z$. For $v\in E^u_z\oplus E^s_z$, let $|v|=|v|_z$ be the max norm which is defined before. Let $A_z$ be an affine automorphism of $\mathbb{R}^2$ such that
\begin{itemize}
\item $A_z(z)=z$
\item $D_zA\begin{bmatrix}
1 \\
0
\end{bmatrix}=\begin{bmatrix}
1 \\
a(z)
\end{bmatrix}\in E^u_z$
\item $D_zA\begin{bmatrix}
0 \\
1
\end{bmatrix}=\begin{bmatrix}
b(z) \\
1
\end{bmatrix}\in E^s_z.$
\end{itemize}
Note that $|a(z)|,|b(z)|\leq \alpha$ and in this case $a(z)=\alpha(z) F_{2x}(z)$ and $b(z)=\beta(z) F_{1y}(z)$, where $|\alpha(z)|,|\beta(z)|\leq1$. Let $\tilde{F}^{-1}$ be the local representation of $F^{-1}$ in this coordinate which means $\tilde{F}^{-1}=A_{F^{-1}(z)}^{-1}F^{-1}A_z$. Then, the matrix $D\tilde{F}^{-1}(z)$ is diagonal. Let
$$D\tilde{F}^{-1}=\begin{bmatrix}
\tilde{g}_{1x} &\tilde{g}_{1y}\\
\tilde{g}_{2x} & \tilde{g}_{2y}
\end{bmatrix}=A_{F^{-1}(z)}^{-1}DF^{-1}A_z.$$
So, for any $w\in U_{a_i}$, one gets
\begin{itemize}
\item
$J_FJ_{A_{F^{-1}(z)}}\tilde{g}_{1x}(w)=F_{2y}+b({F^{-1}(z)})F_{2x}-a(z)F_{1y}-a(z)b({F^{-1}(z)})F_{1x}$,
\item
$J_FJ_{A_{F^{-1}(z)}}\tilde{g}_{1y}(w)=b(z)F_{2y}+b(z)b({F^{-1}(z)})F_{2x}-F_{1y}-b({F^{-1}(z)})F_{1x}$,
\item
$J_FJ_{A_{F^{-1}(z)}}\tilde{g}_{2x}(w)=-a({F^{-1}(z)})F_{2y}-F_{2x}+a(z)a({F^{-1}(z)})F_{1y}+a(z)F_{1x}$,
\item
$J_FJ_{A_{F^{-1}(z)}}\tilde{g}_{2y}(w)=-a({F^{-1}(z)})b(z)F_{2y}-b(z)F_{2x}+a({F^{-1}(z)})F_{1y}+F_{1x}$,
\end{itemize}
where $J_{A_{F^{-1}(z)}}=1-a({F^{-1}(z)})b({F^{-1}(z)})$ and the partial derivatives of $F_1$ and $F_2$ are evaluated at $A_z(w)$.
\begin{lemma}\label{4.1}
Under the notations above, there exist positive constants $C_2$, $C_3$ and $C_4$ such that for any point $w$ close to $z$ lying on the same unstable curve, we have
\begin{equation}\label{9}
|\tilde{g}_{1x}(z)|\leq C_2,
\end{equation}
\begin{equation}\label{10}
|\tilde{g}_{2y}(w)|\geq C_3,
\end{equation}
\begin{equation}\label{11}
\dfrac{|\tilde{g}_{2x}(w)|}{|\tilde{g}_{2y}(w)|}\leq C_4.
\end{equation}
These constants are independent of the choice of $z$ and $w$. Also,
\begin{equation}\label{12}
\dfrac{|\tilde{g}_{1x}(w)|}{|\tilde{g}_{2y}(w)|}\leq \frac{1}{K_0^2}.
\end{equation}
\end{lemma}
\begin{proof}
Due to the above conditions and the explicit formula of $\tilde{g}_{1x}$, it only suffices to estimate the value of
$$H(z):=\frac{a(z)b({F^{-1}(z)})F_{1x}(z)}{J_F(z)J_{A_{F^{-1}(z)}}}.$$
As mentioned before, $|a(z)|\leq |F_{2x}(z)|$ and $|b(z)|\leq|F_{1y}(z)|$. So,
\begin{align*}
|H(z)|&\leq \frac{|(F_{1y}(z)F_{2x}(F^{-1}(z)))F_{1x}(z)|}{|(F_{1x}(z)F_{2y}(z)-F_{1y}(z)F_{2x}(z))(1-F_{1y}(F^{-1}(z))F_{2x}(F^{-1}(z)))|}\\
&\leq 2 \frac{|F_{1y}(z)F_{2x}(F^{-1}(z))|}{|F_{2y}(z)|}.
\end{align*}
The last inequality holds by \ref{A4} and the last term is bounded due to \ref{A3}. This proves \eqref{9}.

For \eqref{10} and \eqref{11}, we claim that $|\tilde{g}_{2y}(w)|\geq C|F_{1x}(z)|$, for some positive constant $C$. First, note that $D_{A_z}F^{-1}\big( A_z\begin{bmatrix}0 \\ 1\end{bmatrix}\big)$ belongs to the stable cone $\mathcal{C}_\alpha^s$ and so is a multiple of some vector $\begin{bmatrix}b \\ 1\end{bmatrix}$ for $|b|\leq \alpha$. Thus,
$$D_w\tilde{F}^{-1}\begin{bmatrix}0 \\ 1\end{bmatrix}=\begin{bmatrix}\tilde{g}_{1y}(w) \\ \tilde{g}_{2y}(w)\end{bmatrix}$$
is a multiple of
$$\begin{bmatrix}b-b({F^{-1}(z)}) \\ -ba({F^{-1}(z)})+1\end{bmatrix},$$

and hence,
$$\left|\begin{bmatrix}\tilde{g}_{1y}(w) \\ \tilde{g}_{2y}(w)\end{bmatrix}\right|=\max\left\{|\tilde{g}_{1y}(w)|,|\tilde{g}_{2y}(w)|\right\}\leq |\tilde{g}_{2y}(w)|\max \left\{\frac{2\alpha}{1-\alpha^2},1\right\}.$$
Since $A_{F^{-1}(z)}$ is uniformly bounded, using \eqref{6} and \eqref{8}, one gets
\begin{align*}
\left|D_w\tilde{F}^{-1}\begin{bmatrix}0 \\ 1\end{bmatrix}\right|&\geq K\left| D_{A_z(w)}F^{-1}\begin{bmatrix}b(z) \\ 1\end{bmatrix}\right|\\
&\geq K\left|\frac{1}{J_F(z)}(-b(z)F_{2x}(A_z(w))+F_{1x}(A_z(w)))\right|\\
&\geq K'\left(|F_{1x}(A_z(w))|-\alpha^2 |F_{1x}(A_z(w))|\right)\geq K''|F_{1x}(z)|,
\end{align*}
for some positive constants $K$, $K'$ and $K''$. This proves \eqref{10}, because $\inf_{z\in \tilde{S}}|F_{1x}(z)|>1$. On the other hand, by using \eqref{5}, \eqref{6}, \eqref{7} and \eqref{8} in the explicit formula of $\tilde{g}_{2x}(w)$,
$$|\tilde{g}_{2x}(w)|\leq C(\alpha)|F_{1x}(A_z(w))|\leq \bar{C}|F_{1x}(z)|.$$
This proves \eqref{11}.
Using \ref{H2}, one gets
$$\frac{|\tilde{g}_{1x}(z)|}{|\tilde{g}_{2y}(z)|}\leq \frac{1}{K_0},$$
so for $w$ sufficiently close to $z$, equation \eqref{12} holds.
\end{proof}
Now we are ready to prove Proposition \ref{pro4}.
\begin{proof}[Proof of Proposition \ref{pro4}.]
Fix $[\A]_n$. Let $z\in U_{[\A]_n}(x_1)\cap \tilde{S}$ be an arbitrary point and $\gamma$ be a $C^2$ unstable curve through $z$. Let $w=\gamma\cap U_{[\A]_n}(x_2)$. There exist $\tau_{1n}\in U_{[\A]_n}(x_1)$ and $\tau_{2n}\in U_{[\A]_n}(x_2)$ such that for $j=1,2$, the following holds:
$$|F^{-1}_{[\A]_n}(U_{[\A]_n}(x_j))|=|D^sF^{-1}_{[\A]_n}(\tau_{jn})|| U_{[\A]_n}(x_j)|.$$
So, by Proposition \ref{R1}, it suffices to show that
\begin{equation}\label{13}
\bar{K}^{-1}<\frac{|D_z^sF^{-1}_{[\A]_n}|}{|D_w^sF^{-1}_{[\A]_n}|}<\bar{K}
\end{equation}
holds for some $\bar{K}$.
\begin{figure}
\def\svgwidth{7cm}
\includegraphics{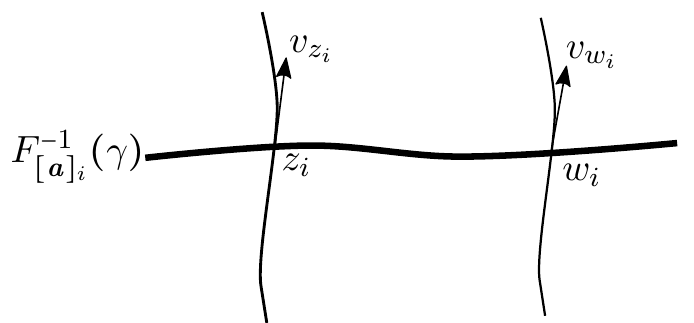}
\centering
\caption{Position of $z_i$, $w_i$, $v_{z_i}$ and $v_{w_i}$}
\label{GH2}
\end{figure}
Let $z_i=F^{-1}_{[\A]_i}(z)$ and $w_i=F^{-1}_{[\A]_i}(w)$, for $i=1,\ldots,n$. As before, we use the affine coordinates induced by splitting $\mathbb{R}^2=E^u_{z_i}\oplus E^s_{z_i}$ at $z_i$, where $E^u_{z_i}$ contains the tangent vector to $F^{-1}_{[\A]_i}(\gamma)$ at $z_i$ and $E^s_{z_i}$ is tangent to the stable manifold at $z_i$. Let $\tilde{F}_{a_i}^{-1}$ be the representation of  $F_{a_i}^{-1}$ in this coordinates and $\tilde{B}_i$ be the small parallelogram centered at $z_i$ using the max norm. One has that
$$\frac{|D_z^sF^{-1}_{[\A]_n}|}{|D_w^sF^{-1}_{[\A]_n}|}\leq \const \prod_{i=1}^{n}\frac{|D_{z_i}\tilde{F}_{a_i}^{-1}(v_{z_i})|}{|D_{w_i}\tilde{F}_{a_i}^{-1}(v_{w_i})|},$$
where $v_{z_i}$ and $v_{w_i}$ are the unit tangent vector to $F^{-1}_{[\A]_i}(U_{[\A]_i}(x_1))$ at $z_i$ and $F^{-1}_{[\A]_i}(U_{[\A]_i}(x_2))$ at $w_i$, respectively (see Figure \ref{GH2}). For proving \eqref{13}, it is needed to show that
$$\sum_{i=1}^{n}\log |D_{z_i}\tilde{F}_{a_i}^{-1}(v_{z_i})|-\log |D_{w_i}\tilde{F}_{a_i}^{-1}(v_{w_i})|$$
is uniformly bounded. By the H{\"o}lder continuity of logarithm, there exists $C>0$ such that the above term is less than $C$ times of the following quantity,
%\begin{align*}
$$\sum_{i=1}^{n}|D_{z_i}\tilde{F}_{a_i}^{-1}(v_{z_i})-D_{w_i}\tilde{F}^{-1}(v_{w_i})|\leq \sum_{i=1}^{n}|D_{z_i}\tilde{F}_{a_i}^{-1}(v_{z_i}-v_{w_i})|\\+\sum_{i=1}^{n}|D_{z_i}\tilde{F}_{a_i}^{-1}-D_{w_i}\tilde{F}_{a_i}^{-1}||v_{w_i}|.$$
Since $|v_{w_i}|=1$, using the mean value theorem, the second term on the right is less than
$$\max_{z\in\tilde{S}} |D^2\tilde{F}_{a_i}^{-1}|\sum_{i=1}^n |z_i-w_i|\leq
\max_{z\in\tilde{S}} |D^2\tilde{F}_{a_i}^{-1}|\sum_{i=1}^\infty |z_i-w_i|,$$
which is bounded, independent of $n$, since $|z_i-w_i|\leq \const d_u(z_i,w_i)\leq\const K_0^{-i}$, where $d_u$ is the metric induced by the Riemannian metric on the unstable curve. Hence, it is suffices to show that
$$\sum_{i=1}^{n}|D_{z_i}\tilde{F}_{a_i}^{-1}(v_{z_i}-v_{w_i})|$$
is bounded and the bound is independent of $n$.

In these affine coordinates, one has that $v_{z_i}=\begin{bmatrix}0\\1 \end{bmatrix}$. Also, for sufficiently large $i$, $v_{w_i}\in \mathcal{C}^s_{C_4}$, where $C_4$ is from Lemma \ref{4.1}. Let $v_{w_i}=(u_{1,i},u_{2,i})$ and $D_{w_{i-1}}\tilde{F}_{a_{i-1}}^{-1}(v_{w_{i-1}})=(\xi_i,\eta_i)$. So,
\begin{align*}
\xi_i&=\tilde{g}_{1x}(w_{i-1})u_{1,i-1}+\tilde{g}_{1y}(w_{i-1})u_{2,i-1}\\
\eta_i&=\tilde{g}_{2x}(w_{i-1})u_{1,i-1}+\tilde{g}_{2y}(w_{i-1})u_{2,i-1}.
\end{align*}
Since $|D_{w_{i-1}}\tilde{F}^{-1}(v_{w_{i-1}})|=|\eta_i|$, $u_{1,i}=\xi_i/|\eta_i|$ and $u_{2,i}=1$. Then, by \eqref{9}, one has that
$$|D_{z_i}\tilde{F}^{-1}(v_{z_i}-v_{w_i})|=|\tilde{g}_{1x}(z_i)|\frac{|\xi_i|}{|\eta_i|}\leq C_2\frac{|\xi_i|}{|\eta_i|}.$$
Now, we try to bound the sum of the last fraction. By \eqref{11}, for sufficiently large $i$, one has that
\begin{align*}
|\eta_i |=|\tilde{g}_{2x}(w_{i-1})u_{1,i-1}+\tilde{g}_{2y}(w_{i-1})u_{2,i-1}|=&|\tilde{g}_{2y}(w_{i-1})|\left|\frac{\tilde{g}_{2x}(w_{i-1})}{\tilde{g}_{2y}(w_{i-1})}
\frac{u_{1,i-1}}{u_{2,i-1}}+1\right|\\\geq& |\tilde{g}_{2y}(w_{i-1})|(1-C_4^2).
\end{align*}
So,
$$|u_{1,i}|=\frac{|\xi_i|}{|\eta_i|}\leq \frac{1}{1-C_4^2}\left(\frac{|\tilde{g}_{1x}(w_{i-1})|}{|\tilde{g}_{2y}(w_{i-1})|}|u_{1,i-1}|+\frac{|\tilde{g}_{1y}(w_{i-1})|}{|\tilde{g}_{2y}(w_{i-1})|}\right).$$
On the other hand, by \eqref{11} and equality $\tilde{g}_{1y}(z_{i-1})=0$, for sufficiently large $i$ there exists $\tau_{i-1}$ such that
\begin{align*}
\frac{|\tilde{g}_{1y}(w_{i-1})|}{|\tilde{g}_{2y}(w_{i-1})|}&\leq \frac{|\tilde{g}_{1yx}(\tau_{i-1})|}{|\tilde{g}_{2y}(w_{i-1})|}|w_{i-1}-z_{i-1}|+\frac{|\tilde{g}_{1yy}(\tau_{i-1})|}{|\tilde{g}_{2y}(w_{i-1})|}|w_{i-1}-z_{i-1}|\\
&\leq \frac{2}{C_4}\max_{\tau\in\tilde{S}}|D^2\tilde{F}^{-1}(\tau)||w_{i-1}-z_{i-1}|.
\end{align*}
Then, by \eqref{12} and inequality $|z_{i}-w_{i}|\leq \const K_0^{-i}$, one has that
$$\frac{|\xi_i|}{|\eta_i|}\leq \frac{1}{(1-C_4^2)K_0^2}|u_{1,i-1}|+\tilde{C}\left(\frac{1}{K_0}\right)^{i-1}.$$
Supposing inductively that $|u_{1,i-1}|\leq 2\tilde{C}\left(\frac{1}{K_0}\right)^{i-2}$, one gets
$$|u_{1,i}|=\frac{|\xi_i|}{|\eta_i|}\leq \frac{2\tilde{C}}{(1-C_4^2)K_0^2}\left(\frac{1}{K_0}\right)^{n-2}+\tilde{C}\left(\frac{1}{K_0}\right)^{i-1}.$$
Assuming that $C_4<\frac{1}{4}$ and $K_0>3$, one gets
$$|u_{1,i}|\leq 2\tilde{C}\left(\frac{1}{K_0}\right)^{i-1}.$$
Therefore,
$$\sum_{i=1}^n |D_{z_i}\tilde{F}^{-1}(v_{z_i}-v_{w_i})|\leq C_2\sum_{i=1}^n |u_{1,i}|<2C_2\tilde{C}\sum_{i=1}^\infty\left(\frac{1}{K_0}\right)^{i-1}.$$
The last sum converges and this finishes the proof of Proposition \ref{pro4}.
\end{proof}
\begin{remark}\label{R2}
Fix two finite words $[\A]_n$ and $[\B]_m$,
\begin{itemize}
\item there exists a constant $K_3$ such that the components of $\hat{U}_{[\A]_n}(x)\setminus U_{[\A]_n}(x)$ have length not smaller than $K_3\cdot |\hat{U}_{[\A]_n}(x)|$,
\item there exists a constant $K_4$ such that
$$K_4^{-1}\leq \frac{d([\A]_n[\B]_m)}{d([\A]_n)d([\B]_m)}\leq K_4.$$
\end{itemize}
\end{remark}
The first part of the remark above is a simple application of the classical bounded distortion property. The second part follows from Proposition \ref{pro4} and the fact that $F_{[\A]_n[\B]_m}=F_{[\B]_m}\circ F_{[\A]_n}$.

Let $B^s_r(z)=\{w\in W^s(z) \;|\; d_s(z,w)<r\}$ for $r\in (0,+\infty)$. Combining Proposition \ref{pro4} and Remark \ref{R2}, one gets the following corollary:
\begin{corollary}\label{c1}
If $r<K_3K_2^{-1}d([\A]_n)$ and $B^s_r(z)$ intersects $U_{[\A]_n}$ then $z\in \hat{U}_{[\A]_n}$.
\end{corollary}
\subsection{Proof of Theorem \ref{T1}}
The analytical approach in \cite{T} needs the explicit formula of $F$ and it is not applicable here, although the sufficient condition for the SRB measure $\mu_F$ to be absolutely continuous prepared in \cite{T} is still the key tool to prove Theorem \ref{T1}.

For the SRB measure $\mu_F$, there exist probability measures $\mu_x$ along the stable manifolds $W^s(x)$ which we may write
$$\mu_F=\int \mu_x d\mu_g(x).$$
Tsujii proved the following remarkable proposition in \cite{T}.
\begin{proposition}
For a positive real number $r$, let
$$\Vert \mu_x\Vert_r^2=\int_\mathbb{R} (\mu_x(B_r^s(z)))^2dz$$
and
$$I(r)=r^{-2}\int_0^1\Vert \mu_x\Vert_r^2dx.$$
If $\liminf_{r\rightarrow0}I(r)<\infty$ then $\mu_F$ is absolutely continuous with respect to the Lebesgue measure and its density function is square integrable.
\end{proposition}
We use the geometric approach of Rams in \cite{R} in order to prove that the SRB measure is an ACIP.
%For all closed curves $\gamma$ in the stable manifold $W^s(x)$ through $(x,0)$, we have $$\mu_x(\gamma)=\lim_{n\to\infty}\sum_{[\A]_n;\gamma\cap U_{[\A]_n}(x)\neq\emptyset}\mu_F(\hat{S}_{[\A]_n})=\const \lim_{n\to\infty}\sum_{[\A]_n;\gamma\cap U_{[\A]_n}(x)\neq\emptyset}\mathrm{diam} I_{[\A]_n}.$$
Denote the inverse branches of $g$ by $g_i$ for $i\in \mathbb{N}$ and let $g_{[\A]_n}:=g_{a_1}\circ \cdots \circ g_{a_n}$. Due to Adler's Theorem \cite{A} and \ref{A1}, there exist positive constants $l$ and $L$ such that
$$l<\frac{d\mu_g}{dx}<L,$$
so clearly the following inequalities hold
\begin{equation}\label{14}
\frac{l}{L}|I_{[\A]_n}|\leq \frac{d}{dx}g_{[\A]_n}\leq \frac{L}{l}|I_{[\A]_n}|
\end{equation}
and
\begin{equation}\label{15}
|I_{[\A]_n[\B]_m}|\leq\frac{L}{l}|I_{[\A]_n}| |I_{[\B]_m}|.
\end{equation}
\begin{remark}\label{l1}~
\begin{itemize}
\item
For any two $\delta$-transversal finite words $[\A]_n$ and $[\B]_m$, any two subword $[\A]_j^i=(a_i,\ldots,a_j)$ and $[\B]_l^k=(b_k,\ldots,b_l)$ are $\delta$-transversal.
\item
For any $0<c_1<c_2$,
$$\sum_{[\A]_n;c_1<d([\A]_n)<c_2}|I_{[\A]_n}|\leq 1+\frac{\log c_2-\log c_1}{\log m},$$
where $m$ is the infimum of $g'$ on $\bigcup_{i\in\mathbb{N}}I_i$. The inequality holds since an interval of length $c_2$, contains at most $(\log c_2-\log c_1)/\log m$ disjoint intervals of length $c_1$.
\end{itemize}
\end{remark}
Now, we are prepared to prove Theorem \ref{T1}.
\begin{proof}[Proof of Theorem \ref{T1}]
For sufficiently small $r$, let $N(r)\in \mathbb{N}$ be the biggest number such that for any $i\in\{1,\ldots,N(r)\}$, $\inf_{z\in S_i}D^sF(z)>r$. For any $\A\in \Sigma_{N(r)}^\infty=\{(a_i)_{i=1}^\infty|a_i\in\{1,\ldots,N(r)\}\}$, let cylinder $Z_{[\A]_n}$ be the set of all words in $\Sigma_{N(r)}^\infty$ that begin with $[\A]_n$. Then, the cylinders $\{Z_{[\A]_n}; [\A]_n\in M(r)\}$ form a finite disjoint cover of $\Sigma_{N(r)}^\infty$, and
\begin{equation}\label{16}
\lim_{r\rightarrow0}\sum_{[\A]_n\in M(r)}|I_{[\A]_n}|=1.
\end{equation}
Since $\mu_F$ is invariant, for any $x\in[0,1]$, $z\in W^s(x)$ and $r>0$,
\begin{equation}\label{Inv}
\mu_x(B_r^s(z))=\lim_{n\to\infty}\sum_{\substack{[\A]_n;B_r^s(z)\cap U_{[\A]_n}(x)\neq\emptyset\\ y\in I_{[\A]_n},g^n(y)=x }}\frac{\mu_y\big(F_{[\A]_n}^{-1}(B_r^s(z))\big)}{(g^n)'(y)}\leq\lim_{n\to\infty}\sum_{[\A]_n;B_r^s(z)\cap U_{[\A]_n}(x)\neq\emptyset} |I_{[\A]_n}|.
\end{equation}
Fix a small positive $r$ and let $R=K_3^{-1}K_2r$. According to Corollary \ref{c1}, \eqref{14} and \eqref{Inv}, one has that
$$\mu_x(B_r^s(z))\leq L\sum_{[\A]_n\in M(R);z\in\hat{U}_{[\A]_n}(x)} |I_{[\A]_n}|.$$
Then
$$\Vert\mu_x\Vert_r^2\leq L^2\sum_{[\A]_n\in M(R)}\sum_{[\B]_m\in M(R)}|(\hat{U}_{[\A]_n}(x)\cap\hat{U}_{[\B]_m}(x)| |I_{[\A]_n}| |I_{[\B]_m}|$$
and
\begin{equation}\label{17}
I(r)\leq r^{-2}L^2\sum_{[\A]_n\in M(R)}\sum_{[\B]_m\in M(R)}\mathrm{vol}(\hat{U}_{[\A]_n}\cap\hat{U}_{[\B]_m})|I_{[\A]_n}| |I_{[\B]_m}|.
\end{equation}
Now, for $\delta>0$ put
$$I^{Tr}_\delta(r)=r^{-2}L^2\sum_{([\A]_n,[\B]_m)\in {\bf M}_\delta^{Tr}(R) }\mathrm{vol}(\hat{U}_{[\A]_n}\cap\hat{U}_{[\B]_m}) |I_{[\A]_n}| |I_{[\B]_m}|$$
and similarly put $I^{NTr}_\delta(r)$ for the sum corresponds to  ${\bf M}_\delta^{NTr}(R)$. 
%=r^{-2}L^2\sum_{([\A]_n,[\B]_m)\in {\bf M}_\delta^{NTr}(R)}\mathrm{vol}(\hat{U}_{[\A]_n}\cap\hat{U}_{[\B]_m}) |I_{[\A]_n}||I_{[\B]_m}|.$$
By (\ref{17}), $I(r)\leq I_\delta^{Tr}(r)+ I_\delta^{NTr}(r)$. In view of the transversality condition, there is $\delta_0$ such that $I^{NTr}_{\delta_0}(r)<\infty$.
Hence, to bound $I(r)$, it is sufficient to bound $I_{\delta_0}^{Tr}(r)$. Following Rams \cite{R}, for $i\geq0$, put
\begin{equation}\label{18}
I_i(r)=r^{-2}L^2\sum_{[\A]_n}\sum_{[\B]_m}\sum_{[\C]_i}\mathrm{vol}(\hat{U}_{[\C]_i[\A]_n}\cap\hat{U}_{[\C]_i[\B]_m}) |I_{[\C]_i[\A]_n}||I_{[\C]_i[\B]_m}|,
\end{equation}
where the sum is taken over such words that $([\C]_i[\A]_n,[\C]_i [\B]_m)\in {\bf M}^{Tr}_{\delta_0}(R)$ and $a_1\neq b_1$. So,
$$I_{\delta_0}^{Tr}(r)= I_0(r)+\cdots+I_i(r)+\cdots.$$
According to Remark \ref{l1}, \eqref{16} and \ref{A2}, one gives
$$\sum_{[\A]_n\in M(R)}\sum_{[\B]_m\in M(R)}|I_{[\A]_n}||I_{[\B]_m}|<1,$$
and so $I_0(r)\leq K_5$ for some constant $K_5$.

Now, fix $i>0$. By Remark \ref{R2} and the definition of $M(R)$, one has that
\begin{equation}\label{19}
\frac{R}{K_4d([\C]_i)}\leq d([\A]_n)\leq \frac{K_4R}{d([\C]_i)}.
\end{equation}
The above inequality also holds for $d([\B]_m)$. Hence
\begin{equation}\label{20}
K_4^{-2}\leq \frac{d([\A]_n)}{d([\B]_m)}\leq K_4^{2}
\end{equation}
From \eqref{15}, \eqref{18} and \eqref{19}, we have
$$I_i(r)\leq \sum_{[\A]_n}\sum_{[\B]_m}\sum_{[\C]_i}L^4K_3^2K_2^{-2}K_4^2l^{-2}\frac{|I_{[\A]_n}||I_{[\B]_m}|(| I_{[\C]_i}|)^2}{d([\A]_n)d([\B]_m)(d([\C]_i))^2}\mathrm{vol}(\hat{U}_{[\C]_i[\A]_n}\cap\hat{U}_{[\C]_i[\B]_m}).$$
%and similarly for $\tilde{I}_i(r)$.\\
For fixed $[\A]_n$ and $[\B]_m$, one has that
$$\bigcup_{[\C]_i}(\hat{U}_{[\C]_i[\A]_n}\cap\hat{U}_{[\C]_i[\B]_m})=\bigcup_{[\C]_i}F_{[\C]_i} (\hat{U}_{[\A]_n}\cap\hat{U}_{[\B]_m}),$$
since the contraction of each map $F_{[\C]_i}$ along the vertical direction is at least $K_4d([\C]_i)/|J|$ times and its expansion along the horizontal direction is at most $Ll^{-1}/|I_{[\C]_p}|$. Therefore
\begin{equation}\label{21}
\sum_{[\C]_i}\frac{|I_{[\C]_i}|}{d([\C]_i)}\mathrm{vol}(\hat{U}_{[\C]_i[\A]_n}\cap\hat{U}_{[\C]_i[\B]_m})\leq \frac{K_4L}{l| J|}\mathrm{vol}(\hat{U}_{[\A]_n}\cap\hat{U}_{[\B]_m}).
\end{equation}
So, \eqref{21} and Proposition \ref{R1} imply that
$$I_i(r)\leq \sum_{[\A]_n}\sum_{[\B]_m} \frac{L^5K_2^2K_4^3}{\delta_0l^3K_3^2 |J|}|I_{[\A]_n}||I_{[\B]_m}|\sup\frac{| I_{[\C]_i}|}{d([\C]_i)}.$$
Let
$$\underline{\eta_i}=\inf_{z\in \bigcup_{j=1}^{N(R)}S_j}D^s  F_{[\C]_i}(z) \quad \text{and} \quad \overline{\eta_i}=\sup_{z\in \bigcup_{j=1}^{N(R)}S_j}D^s  F_{[\C]_i}(z).$$
By bounded distortion property, Proposition \ref{R1}, we have $K^{-1}\leq\overline{\eta_i}/\underline{\eta_i}\leq K$ (where $K$ is independent of the choice of $[\C]_i$). So by \eqref{19} and the fact that $\underline{\eta_i} |J|\leq d([\C]_i)\leq \overline{\eta_i}|J|$, and Remark \ref{l1}, we get
$$\sum_{\substack{[\C]_i[\A]_n\in M(R)\\ \text{for some}\; [\C]_i}} |I_{[\A]_n}|=1+\frac{2\log K_4+\log \overline{\eta_i}-\log \underline{\eta_i}}{\log m} \leq 1+\frac{2\log K_4+\log K}{\log m},$$
and the same holds for $I_{[\B]_m}$. Therefore
$$I_i(r)\leq \frac{L^5K_3^2K_4^3}{\delta_0l^3K_2^2|J|}\left(1+\frac{2\log K_4+\log K}{\log m}\right)^2\sup\frac{|I_{[\C]_i}|}{d([\C]_i)}.$$
The fatness condition implies that
$$\frac{|I_{[\C]_i}|}{d([\C]_i)}\leq K_1 (d([\C]_i))^\epsilon\leq K_1|J|^\epsilon M^{i\epsilon},$$
where $M=\sup_{z\in\Lambda}D^sF(z)<1$, so
$$I_i(r)\leq \const M^{i\epsilon},$$
and thus \eqref{17} are uniformly summable.
\end{proof}

\end{document}